\documentclass[12pt,leqno]{article}

\usepackage{fullpage}
\usepackage{amssymb}
\usepackage{amsmath}
\usepackage{amsthm}
\newtheorem{theorem}{Theorem}
\newtheorem{problem}{Problem}

\newtheorem{definition}{Definition}

\newtheorem{lemma}{Lemma}

\begin{document}
\title{\large\bf On a Problem of Erd\H os, Herzog and Sch\"{o}nheim\footnote{This work was supported by the National Natural
Science Foundation of China, Grant No 11071121. }}
\date{}
\author{ Yong-Gao Chen\footnote{Corresponding author, email: ygchen@njnu.edu.cn}  and
Cui-Ying Hu
\\
\small School of Mathematical Sciences and Institute of Mathematics,
\\ \small Nanjing Normal University,  Nanjing 210046, P. R. CHINA}
 \maketitle

\begin{abstract} Let $p_1, p_2, \dots , p_n$ be distinct primes.
 In 1970, Erd\H os, Herzog and Sch\"{o}nheim proved that if $\cal D$ is a set of divisors of
$N=p_1^{\alpha_1}\cdots p_n^{\alpha_n}$, $\alpha_1\ge
\alpha_2\ge\cdots \ge \alpha_n$, no two members of the set being
coprime and if no additional member may be included in $\cal D$
without contradicting this requirement then $ |{\cal D}|\ge
\alpha_n \prod_{i=1}^{n-1} (\alpha_i +1) $. They asked to
determine all sets $\cal D$ such that the equality holds. In this
paper we solve this problem. We also pose several open problems
for further research.

\end{abstract}



\section{ Introduction}

Many theorems on intersections of sets have been established. One of
intersection theorems is the next theorem of  Erd\H os, Ko and Rado.

{\bf Theorem A \cite[Erd\H os-Ko-Rado]{Erdos-Ko}.} {\it If ${\cal
A} =\{ A_1, A_2, \ldots , A_m\} $ is a family of (different)
subsets of a given set $M$, $|M|=n$, such that $A_i\cap A_j\not=
\emptyset $ for every $i,j$, then

a) $m\le 2^{n-1}$ and for every $n$ there are $m=2^{n-1}$ such
subsets;

b) if $m<2^{n-1}$ then additional members may be included in $\cal
A$, the enlarged family still satisfying $A_i\cap A_j\not= \emptyset
$ for every $i,j$.}\vskip 3mm

Theorem A is equivalent to the following theorem.\vskip 3mm

{\bf Theorem B.} {\it If ${\cal A} =\{ d_1, d_2, \ldots , d_m\} $
is a set of (different) divisors of a given positive integer $N$,
$N=p_1p_2\cdots p_n$, where $p_1, p_2, \ldots , p_n$ are distinct
primes, such that $(d_i, d_j)>1 $ for every $i,j$, then

a) $m\le 2^{n-1}$ and for every $n$ there are $m=2^{n-1}$ such
divisors;

b) if $m<2^{n-1}$ then additional members may be included in $\cal
A$, the enlarged set still satisfying $(d_i, d_j)>1 $ for every
$i,j$.}\vskip 3mm

This means that if $\cal A$ is a maximal set with the property
$(d_i, d_j)>1 $ for every $i,j$, then $|{\cal A}|=2^{n-1}$. If we
allow repetitions in $M$ (resp. $N$ is not squarefree), it is more
convenient to state results with the language of divisors (see
\cite{Erdos-Herzog}, \cite{Erdos-Schonheim} and
\cite{Herzog-Schonheim}).

 In this paper, $p_1, p_2, \ldots , p_n$
are always distinct primes. Erd\H os, Herzog and Sch\"{o}nheim
\cite{Erdos-Herzog} proved the following theorem.\vskip 3mm

{\bf Theorem C \cite[Erd\H os-Herzog-Sch\"{o}nheim]{Erdos-Herzog}.
} {\it If $\cal D$, $|{\cal D}|=m$, is a set of divisors of
$N=p_1^{\alpha_1}\cdots p_n^{\alpha_n}$, $\alpha_1\ge
\alpha_2\ge\cdots \ge \alpha_n$, no two members of the set being
coprime and if no additional member may be included in $\cal D$
without contradicting this requirement then
\begin{equation*}\label{EHS} m\ge \alpha_n \prod_{i=1}^{n-1} (\alpha_i
+1) .\end{equation*} }\vskip 3mm

If $\cal D$ is the set of all positive divisors of $N$ which are
divisible by $p_n$, then $\cal D$ satisfies the assumptions of
Theorem C and has the minimum size, that is,
\begin{equation}\label{EHS} |{\cal D}|=\alpha_n \prod_{i=1}^{n-1} (\alpha_i
+1) .\end{equation} In \cite[Final remark]{Erdos-Herzog}, Erd\H
os, Herzog and Sch\"{o}nheim remarked that it would be of interest
to determine all sets $\cal D$ satisfying the assumptions of
Theorem C with (1).

In this paper we solve this problem. For convenience, we introduce
the following definitions.

\begin{definition}\label{def1} A set $\cal D$ of positive divisors of $N$ is
a $N-$set if no two elements of the set are coprime. A $N-$set
$\cal D$ is maximal if no additional divisor of $N$ may be
included.\end{definition}

\begin{definition}\label{def2} For a set $\cal D$ of
positive divisors of $N=p_1^{\alpha_1}\cdots p_n^{\alpha_n}$, an
element $d$ of $\cal D$ is a divisible minimal element if $d$ cannot
be divided by any other element of $\cal D$. Denote by $d({\cal D})$
the set of all divisible minimal elements of $\cal D$.
\end{definition}

It is clear that if $\cal D$ is a maximal $N-$set and $d\in {\cal
D}$, then $l\in {\cal D}$ for all $l\mid N$ with $d\mid l$. Now
the above Erd\H os-Herzog-Sch\"{o}nheim problem can be restated as

\begin{problem} Let $N=p_1^{\alpha_1}\cdots p_n^{\alpha_n}$,
where  $\alpha_1\ge \cdots \ge \alpha_n>0$. Determine all maximal
$N-$sets $\cal D$ with the minimum size.\end{problem}

First we find some maximal $N-$sets $\cal D$ with the minimum size.
Let
$$\alpha_1\ge \cdots\ge \alpha_u > \alpha_{u+1}=\cdots = \alpha_n.$$
If $\alpha_1=\cdots =\alpha_n$, let $u=0$. For any $v$ with $1\le
v\le n$, let
$${\cal D}(p_v)=\{ d : d\mid N, \, p_v\mid d\} .$$
Then all ${\cal D}(p_v) (1\le v\le n)$ are maximal $N-$sets. For $
u+1\le v\le n$ we have
\begin{equation*}|{\cal D}(p_v)|=\alpha_v
\prod_{i=1,i\not= v}^{n} (\alpha_i +1)=\alpha_n \prod_{i=1}^{n-1}
(\alpha_i +1) .\end{equation*} For $v\le u$ we have
\begin{equation*}|{\cal D}(p_v)|=\alpha_v
\prod_{i=1,i\not= v}^{n} (\alpha_i +1)>\alpha_n \prod_{i=1}^{n-1}
(\alpha_i +1).\end{equation*}

Now we consider the special case $\alpha_n=1$.  Let ${\cal D}'$ be
a maximal $p_{u+1}\cdots p_n-$set. By Theorem B we have  $|{\cal
D}'|=2^{n-u-1}$. Let
$${\cal D}=\left\{ dd' :  d\mid \frac N{p_{u+1}\cdots
p_n}, d'\in {\cal D}'\right\} .$$ Since ${\cal D}'$ is a
$p_{u+1}\cdots p_n-$set, we have $\cal D$ is a $N$-set. For $l\mid
N$ and $l\notin D$, let $l=l_1 l_1'$, where
$$l_1 \mid \frac N{p_{u+1}\cdots
p_n},\quad l_1' \mid p_{u+1}\cdots p_n.$$ By $l\notin D$ we have
$l_1'\notin {\cal D}'$. Since ${\cal D}'$ is a maximal
$p_{u+1}\cdots p_n-$set, there exists $d'\in {\cal D}'$ such that
$(l_1', d')=1$. Thus $(l, d')=1$ and $d'\in {\cal D}$. Thus we have
proved that $\cal D$ is a maximal $N-$set. We have
$$|{\cal D}|=|{ \cal D}'| \prod_{i=1}^u (\alpha_i+1) =2^{n-u-1}\prod_{i=1}^u (\alpha_i+1) =\alpha_n \prod_{i=1}^{n-1}
(\alpha_i +1) .$$

In this paper we show that these are all maximal $N-$sets $\cal D$
with the minimum size.

\begin{theorem}\label{mainthm1}Let $N=p_1^{\alpha_1}\cdots p_n^{\alpha_n}$ with $\alpha_1\ge
\alpha_2\ge\cdots\ge \alpha_u>\alpha_{u+1}=\cdots =\alpha_n\ge 2$.
Then the following statements are equivalent each other:

(a)  $\cal D$ is a maximal $N-$set with the minimum size.

(b) $\cal D$ is a maximal $N-$set with $d({\cal D})=\{ p_v\}$ for
some $u+1\le v\le n$.

(c) ${\cal D} =\{ d : d\mid N, \, p_v\mid d\} $ for some $u+1\le
v\le n$.
\end{theorem}

\begin{theorem}\label{mainthm2}Let $N=p_1^{\alpha_1}\cdots p_n^{\alpha_n}$ with $\alpha_1\ge
\alpha_2\ge\cdots\ge \alpha_u>\alpha_{u+1}=\cdots =\alpha_n=1$.
Then the following statements are equivalent each other:

(a)  $\cal D$ is a maximal $N-$set with the minimum size.

(b) $\cal D$ is a maximal $N-$set with $d({\cal D})\subseteq \{ d :
d\mid p_{u+1}\cdots p_n\} $.

(c) $${\cal D} =\{ dd' : d\mid \frac{N}{p_{u+1}\cdots p_n}, d'\in
{\cal D}'\} $$ for a maximal $p_{u+1}\cdots p_n-$set ${\cal D}'$.
\end{theorem}

 For a set $\cal T$ of positive divisors
of $N$, let $R({\cal T},N)$ be the set of all positive divisors of
$N$ which can be divided by at least one of elements of $\cal T$. It
is easy to see that $R({\cal T},N)$ is a $N-$set if and only if
${\cal T}$ is a $N-$set.

With these notations, we have the following theorems.

\begin{theorem}\label{mainthm5}Let $N=p_1^{\alpha_1}\cdots p_n^{\alpha_n}$ with $\alpha_1\ge
\alpha_2\ge\cdots\ge \alpha_u>\alpha_{u+1}=\cdots =\alpha_n=1$, and
let ${\cal T}_1, \ldots , {\cal T}_k$ be all  sets of positive
divisors of $p_{u+1}\cdots p_n$ such that for each $i$,

(a)  no any two elements of ${\cal T}_i$ are coprime;

(b) no element of  ${\cal T}_i$ can be divided by another element of
 ${\cal T}_i$;

(c) any  divisor of $p_{u+1}\cdots p_n$ is either coprime to some
element of  ${\cal T}_i$  or divisible by one element of ${\cal
T}_i$.

Then $R({\cal T}_1,N), \ldots , R({\cal T}_k,N)$ are all maximal
$N-$sets ${\cal D}$ with the minimum size.
\end{theorem}

{\bf Example.} Let $N=420=2^2\cdot 3\cdot 5\cdot 7$. Then
$p_{u+1}\cdots p_n=3\cdot 5\cdot 7$ and the sets satisfying (a), (b)
and (c) are
$${\cal T}_1=\{ 3\} ,  {\cal T}_2=\{ 5\} , {\cal T}_3=\{ 7\} , {\cal T}_4=\{ 3\cdot 5, 3\cdot 7, 5\cdot
7\} .$$ Thus there are exactly four maximal $420-$sets $R({\cal
T}_1,N), R({\cal T}_2,N), R({\cal T}_3,N), R({\cal T}_4,N)$ such
that the equality in Theorem \ref{mainthm5} holds.

\begin{theorem}\label{mainthm6}Let $N=p_1^{\alpha_1}\cdots p_n^{\alpha_n}$ with $\alpha_1\ge
\alpha_2\ge\cdots\ge \alpha_u>\alpha_{u+1}=\cdots =\alpha_n\ge 2$.
Then $R(\{p_{u+1}\},N), \ldots , R(\{p_n\},N)$ are all maximal
$N-$sets ${\cal D}$ with the minimum size.
\end{theorem}

Theorem \ref{mainthm6} follows from Theorem \ref{mainthm1}
immediately. We pose the following problem.

\begin{problem} Determine the number $H(N)$ of maximal $N-$sets $\cal
D$ with the minimum size.\end{problem}

{\bf Remark.} If $N=p_1^{\alpha_1}\cdots p_n^{\alpha_n}$ with
$\alpha_1\ge \alpha_2\ge\cdots\ge \alpha_u>\alpha_{u+1}=\cdots
=\alpha_n>1$, then by Theorem \ref{mainthm6} we have $H(N)=n-u$.
For the case $N=p_1^{\alpha_1}\cdots p_n^{\alpha_n}$ with
$\alpha_1\ge \alpha_2\ge\cdots\ge \alpha_u>\alpha_{u+1}=\cdots
=\alpha_n=1$, then $H(N)$ is the number of sets  with (a), (b) and
(c) in Theorem \ref{mainthm5}.

\section{Preliminary lemmas}

Let $N=p_1^{\alpha_1}\cdots p_n^{\alpha_n}$, where $\alpha_1\ge
\cdots \ge \alpha_n>0$. Let $N' =p_1\cdots p_n$. For $d\mid N'$,
define
$$\alpha (d)=\prod_{p_i\mid d} \alpha_i, \quad \bar{d} = \frac{N'}{d}.$$
Let
$${\cal A}=\{ d : d\mid N', d\in {\cal D}\} $$
and
$$  {\cal A}_n =\{ d : d\in {\cal A}, p_n\mid d\}, \quad {\cal A}_n'=\{ d : d\in {\cal A}, p_n\nmid d\} . $$
In this section we always assume that $\cal D$ is a maximal
$N-$set. Then $\cal A$ is also a maximal $N'-$set.

\begin{lemma}\label{lem6} If $\{ p_{u+1}, \dots , p_n\} \cap {\cal D} \not= \emptyset $, then
$d({\cal D})=\{ p_v\} $ for some $u+1\le v\le n$.
\end{lemma}

\begin{proof} Since $\{ p_{u+1}, \dots , p_n\} \cap {\cal D} \not= \emptyset
$, there is an integer $u+1\le v\le n$ such that $p_v\in \cal D$.
By $\cal D$ being a $N$-set we have $p_v\mid d$ for all $d\in \cal
D$. Hence $d({\cal D})=\{ p_v\} $. This completes the proof of
Lemma \ref{lem6}.\end{proof}

\begin{lemma}\label{lem1} Let $d\mid N'$. Then exactly one of $d$ and
$\bar d$ is in $\cal A$.\end{lemma}

\begin{proof} Since $(d, \bar d)=1$ and
$\cal A$ is the $N'-$set, we know that at most one of $d$ and $\bar
d$ is in $\cal A$.

Suppose that $d\notin {\cal A}$. By the maximality of $\cal A$ there
exists $d'\in {\cal A}$ such that $(d, d')=1$. Hence $d'\mid \bar
d$. Again, by the maximality of $\cal A$ and $d'\mid \bar d$ we have
$\bar d\in {\cal A}$. This completes the proof of Lemma
\ref{lem1}.\end{proof}

\begin{lemma}\label{lem3} We have
$${\cal A}_n \cup \{ \bar d : d\in {\cal A}_n'\} =\{ l p_n : l\mid
p_1\cdots p_{n-1}\} .$$
\end{lemma}

\begin{proof} It is clear that
${\cal A}_n \cup \{ \bar d : d\in {\cal A}_n'\} \subseteq \{ l p_n :
l\mid p_1\cdots p_{n-1}\} $. Now let $l\mid p_1\cdots p_{n-1}$.
Suppose that $lp_n\notin {\cal A}_n$. Then $lp_n\notin {\cal A}$. By
Lemma \ref{lem1} we have $\overline{lp_n}\in {\cal A}$. Thus
$\overline{lp_n}\in {\cal A}_n'$.     So
$lp_n=\overline{\overline{lp_n}}\in \{ \bar d : d\in {\cal A}_n'\}
$. This completes the proof of Lemma \ref{lem3}.\end{proof}

\begin{lemma}\label{lem4} Let $\cal D$ be a maximal $N-$set with the minimum size and ${\cal
A}_n'=\{ d_1, d_2, \dots , d_s\}$. Then there exists a permutation
$i_1, i_2, \dots , i_s$ of $1, 2, \dots , s$ such that
$$\bar{d_{i_j}} \mid d_jp_n,\quad \alpha (d_j)=\alpha
(\bar{d_{i_j}}),\quad j=1,2,\dots , s.$$
\end{lemma}

\begin{proof}
For $d=p_{i_1}^{\beta_1} \cdots p_{i_k}^{\beta_k}$ with
$0<\beta_j\le \alpha_{i_j} (1\le j\le k)$, by the maximality of
$\cal D$, we have $d\in {\cal D}$ if and only if $p_{i_1} \cdots
p_{i_k}\in {\cal A}$. So
\begin{equation}\label{m2.2}|{\cal D}|=\sum_{d\in {\cal A}} \alpha (d)
=\sum_{d\in {\cal A}_n} \alpha (d)+\sum_{d\in {\cal A}_n'} \alpha
(d).\end{equation} By Lemma \ref{lem3} we have $(\alpha (1)=1)$
\begin{equation}\label{m2.3}
\alpha_n \prod_{i=1}^{n-1} (\alpha_i+1)=\sum_{l\mid p_1\cdots
p_{n-1}} \alpha (lp_n)=\sum_{d\in {\cal A}_n} \alpha (d)+\sum_{d\in
{\cal A}_n'} \alpha (\bar d).\end{equation} Since  $\cal D$ is a
maximal $N-$set with the minimum size, we have
\begin{equation}\label{m2.1}
|{\cal D}|=\alpha_n \prod_{i=1}^{n-1} (\alpha_i+1).\end{equation} By
(\ref{m2.2}), (\ref{m2.3}) and (\ref{m2.1}) we have
\begin{equation}\label{m2.4}
\sum_{d\in {\cal A}_n'} \alpha (d)=\sum_{d\in {\cal A}_n'} \alpha
(\bar d).\end{equation}

In order to prove Theorem C, Erd\H os, Herzog and Sch\"{o}nheim
 proved a combinatorial theorem (\cite[Theorem
3]{Erdos-Herzog}). We will employ its following equivalent form to
prove Lemma \ref{lem4}.

{\bf Theorem D.} {\it Let $M$ be a squarefree integer. Denote by
${\bar {d}}'=M/d$ for $d\mid M$. If $F=\{ d_1, d_2, \ldots , d_s\}
$ is a set of divisors of $M$ such that $d_i\mid d\mid M
\Rightarrow
 d\in F$, then there exists a permutation $i_1, i_2, \dots , i_s$ of $1, 2, \dots , s$ such that $\bar{d_{i_j}}' \mid
 d_j $ $(1\le j\le s)$.}\\

In order to employ Theorem D, let $M=p_1\cdots p_{n-1}$ and
$F={\cal A}_n'$. If $d_i\mid d\mid M$, then by the maximality of
$\cal A$ we have $d\in {\cal A}_n'$. Noting that
$$\bar{d_i}'=\frac
M{d_i}=\frac{N'/d_i}{p_n}=\frac{\bar{d_i}}{p_n},$$ by Theorem D
there exists a permutation $i_1, i_2, \dots , i_s$ of $1, 2, \dots
, s$ such that
$$\frac{\bar{d_{i_j}}}{p_n} \mid
 d_j, \quad 1\le j\le s.$$
 That is, $\bar{d_{i_j}} \mid d_jp_n (1\le j\le s)$.
Let $d_jp_n=\bar{d_{i_j}} e_j $ $(1\le j\le s)$. Since $d_{i_j}\in
 {\cal A}$, by Lemma \ref{lem1} we have $\bar{d_{i_j}}\notin {\cal A}$. Thus
 $\bar{d_{i_j}}/p_n\notin {\cal A} (1\le j\le s)$ by the maximality of $\cal A$.
So $e_j>1$ $(1\le j\le s)$, otherwise, $\bar{d_{i_j}}/p_n=d_j\in
\cal A$, a contradiction. Thus, for $1\le j\le s$, we have
\begin{equation*}\label{m2.5}
\alpha (d_j) \alpha (p_n) =\alpha (d_jp_n)=\alpha
(\bar{d_{i_j}}e_j)= \alpha (\bar{d_{i_j}})
 \alpha (e_j)\ge \alpha (\bar{d_{i_j}})\alpha (p_n).\end{equation*}
Hence
\begin{equation}\label{m2.5} \alpha (d_j)\ge \alpha (\bar{d_{i_j}}),\quad 1\le j\le s.\end{equation}
 By (\ref{m2.4}) and (\ref{m2.5}) we have
\begin{equation*}\label{m2.6}
\alpha (d_j)=\alpha (\bar{d_{i_j}}),\quad 1\le j\le
s.\end{equation*}

This completes the proof of Lemma \ref{lem4}.\end{proof}

\begin{lemma}\label{lem5} We have
$D=R(d(D), N)$.\end{lemma}

\begin{proof} By the maximality of $\cal D$ and $d(D)\subseteq \cal
D$ we have $R(d(D), N)\subseteq \cal D$. By the definition of
$d({\cal D})$ and $R(d(D), N)$ we have ${\cal D} \subseteq R(d(D),
N)$. So $D=R(d(D), N)$. This completes the proof of Lemma
\ref{lem5}.\end{proof}

\section{Proof of Theorems}

\begin{proof}[Proof of Theorem \ref{mainthm1}]

(a) $\Rightarrow$ (b): By Lemma \ref{lem6} we may assume that $\{
p_{u+1}, \dots , p_n\} \cap {\cal D} =\emptyset $. Then $p_n\notin
\cal A$. By Lemma \ref{lem1} we have $\bar{p_n}\in \cal A$. That
is, $\bar{p_n}\in {\cal A}_n'$. Let ${\cal A}_n'=\{ d_1, d_2,
\dots , d_s\}$. By Lemma \ref{lem4} there exists a permutation
$i_1, i_2, \dots , i_s$ of $1, 2, \dots , s$ such that
$$\bar{d_{i_j}} \mid d_jp_n,\quad \alpha (d_j)=\alpha
(\bar{d_{i_j}}).$$ Without loss of generality, we may assume that
$d_{i_1}=\bar{p_n}$. Then $\alpha (d_1)=\alpha
(\bar{d_{i_1}})=\alpha (p_n)=\alpha_n$. Since $\alpha_1\ge
\alpha_2\ge\cdots\ge \alpha_u>\alpha_{u+1}=\cdots =\alpha_n\ge 2$,
we have $d_1\in \{ p_{u+1}, \dots , p_n\}$, a contradiction with $\{
p_{u+1}, \dots , p_n\} \cap {\cal D} =\emptyset $.

(b) $\Rightarrow$ (c): It follows from Lemma \ref{lem5}.

(c) $\Rightarrow$ (a): It follows from the arguments before Theorem
\ref{mainthm1}.

This completes the proof of Theorem \ref{mainthm1}.\end{proof}

\begin{proof}[Proof of Theorem \ref{mainthm2}]

(a) $\Rightarrow$ (b): By Lemma \ref{lem6} we may assume that $\{
p_{u+1}, \dots , p_n\} \cap {\cal D} =\emptyset $. Then $p_n\notin
\cal A$. By Lemma \ref{lem1} we have $\bar{p_n}\in \cal A$. That
is, $\bar{p_n}\in {\cal A}_n'$. Let ${\cal A}_n'=\{ d_1, d_2,
\dots , d_s\}$. By Lemma \ref{lem4} there exists a permutation
$i_1, i_2, \dots , i_s$ of $1, 2, \dots , s$ such that
$\bar{d_{i_j}} \mid d_jp_n,\quad \alpha (d_j)=\alpha
(\bar{d_{i_j}})$. As in Lemma \ref{lem4}, let
$d_jp_n=\bar{d_{i_j}}e_j (1\le j\le s)$. Since $\alpha_n =1$ and
$\alpha (d_j)=\alpha (\bar{d_{i_j}})$, we have $\alpha (e_j)=1$
$(1\le j\le s)$. Hence, for $1\le v\le u$ and $1\le j\le s$ we
have $p_v\nmid e_j$ and
$$p_v\mid d_j\Leftrightarrow p_v\mid
\bar{d_{i_j}} \Leftrightarrow p_v\nmid d_{i_j}.$$ Thus, for $1\le
v\le u$ we have
$$|\{ j : p_v\mid d_j\} | = |\{ j : p_v\nmid d_{i_j}\} |= |\{ j :
p_v\nmid d_j\} |.$$ So, for $1\le v\le u$ we have
\begin{equation}\label{1/2}|\{ j : p_v\mid d_j\} |= |\{ j : p_v\nmid d_{j}\}
|=\frac 12 |{\cal A}_n'| .\end{equation}
 Let $d(D)=\{ h_1, h_2,
\ldots , h_{t}\} $. Then $h_i\nmid h_j$ for all $i\not= j$.
Without loss of generality, we may assume that $p_n\nmid h_i$ $
(1\le i\le r)$ and $p_n\mid h_j$ $(r+1\le j\le t)$. Then each
$d_i\in {\cal A}_n'$ can be divided by at least one of $h_1, h_2,
\ldots , h_r$. Since $\cal D$ is a maximal $N$-set, we have
$d({\cal D})\subseteq \cal A$. So $h_1, h_2, \ldots , h_r\in {\cal
A}_n'$. Fixed $1\le v\le u$. Without loss of generality, we may
assume that $h_1, h_2,\ldots , h_w$ are all $h_i$ with $p_v\nmid
h_i$ and $p_n\nmid h_i$.

Let $ {\cal B}= \{ d : p_v\nmid d, d\in {\cal A}_n' \} $. By
\eqref{1/2} we have
$$|\{ p_vd : d\in {\cal B}\} | =|{\cal B}|
=\frac 12 |{\cal A}_n'|.$$ Since ${\cal B}\cap \{ p_vd : d\in
{\cal B}\} =\emptyset $ and $\{ p_vd : d\in {\cal B}\} \subseteq
{\cal A}_n'$, we have ${\cal A}_n'={\cal B}\cup \{ p_vd : d\in
{\cal B}\} $. Let $d\in {\cal B}$. If $w<i\le r$, then by $p_v\mid
h_i$ we have $h_i\nmid d$. If $r<i\le t$, then by $p_n\mid h_i$
and $d\in {\cal A}_n'$ we have $h_i\nmid d$. That is, $d$ cannot
be divided by any $h_i$ with $i>w$. So $d$ can be divided by one
of $h_1, h_2, \ldots , h_w$. Thus each $d'\in {\cal A}_n'$ can be
divided by one of $h_1, h_2, \ldots , h_w$. Since $w\le r$ and
$h_1, h_2, \ldots , h_r\in {\cal A}_n'$ and $h_i\nmid h_j$ for all
$i\not= j$, we have $w=r$. Thus, we have proved that for all $1\le
v\le u$ we have $p_v\nmid h_i$ $(1\le i\le r)$.

Now we have proved that for any given $i$ with $1\le i\le t$, if
$p_n\nmid h_i$, then $p_v\nmid h_i$ for any $1\le v\le u$. Since
$\alpha_{u+1}=\cdots =\alpha_n=1$, the primes $p_{u+1}, \ldots ,
p_n$ are in the same position. Hence, for any given $i,j$ with
$1\le i\le t$ and $u+1\le j\le n$, if $p_j\nmid h_i$, then
$p_v\nmid h_i$ for any $1\le v\le u$. This means that for $1\le
i\le t$, if $p_{u+1}\cdots p_n\nmid h_i$, then $(p_1\cdots p_u,
h_i)=1$, i.e., $h_i\mid p_{u+1}\cdots p_n$. So, for each $1\le
i\le t$, either $p_{u+1}\cdots p_n\mid h_i$ or $h_i\mid
p_{u+1}\cdots p_n$. Since $h_i\nmid h_j$ for all $i\not= j$, we
have either $p_{u+1}\cdots p_n\mid h_i$ for all $1\le i\le t$ or
$h_i\mid p_{u+1}\cdots p_n$ for all $1\le i\le t$. If
$p_{u+1}\cdots p_n\mid h_i$  for all $1\le i\le t$, then $p_n\mid
h_i$  for all $1\le i\le t$. Thus $p_n\mid d$ for all $d\in {\cal
A}$, a contradiction with $\bar{p_n}\in {\cal A}$ and $p_n\nmid
\bar{p_n}$. Hence $h_i\mid p_{u+1}\cdots p_n$ for all $1\le i\le
t$. That is, $$d(D)\subseteq \{ d : d \mid p_{u+1}\cdots p_n\} .$$

(b) $\Rightarrow$ (c): Let ${\cal D}'={\cal D}\cap \{ d : d \mid
p_{u+1}\cdots p_n\}$. Since $\cal D$ is a $N$-set, ${\cal D}'$ is
a $p_{u+1}\cdots p_n$-set. For $d \mid p_{u+1}\cdots p_n$, if
$d\notin {\cal D}'$, then $d\notin {\cal D}$. Since $\cal D$ is a
maximal $N$-set, there exists $l\in \cal D$ such that $(d, l)=1$.
By the definition of $d(D)$, $l$ can be divided by an element $l'$
of $d(D)$. So $(d, l')=1$. By $d(D)\subseteq  \{ d : d \mid
p_{u+1}\cdots p_n\}$ we have $l'\in {\cal D}'$. Thus we have
proved that ${\cal D}'$ is a maximal $p_{u+1}\cdots p_n$-set. By
 $d(D)\subseteq  \{ d : d \mid p_{u+1}\cdots
p_n\}$ we have $d({\cal D}')=d({\cal D})$. By Lemma \ref{lem5} we
have ${\cal D}'=R(d({\cal D}'),p_{u+1}\cdots p_n)=R(d({\cal
D}),p_{u+1}\cdots p_n)$. Again, by Lemma \ref{lem5} and
$d(D)\subseteq  \{ d : d \mid p_{u+1}\cdots p_n\}$ we have
\begin{eqnarray*} {\cal D}&=&R(d({\cal D}), N)=\{ dd' : d\mid \frac
N{p_{u+1}\cdots p_n}, d'\in R(d({\cal
D}),p_{u+1}\cdots p_n)\}\\
&=&\{ dd' : d\mid \frac N{p_{u+1}\cdots p_n}, d'\in {\cal
D}'\}.\end{eqnarray*}

(c) $\Rightarrow$ (a): It follows from the arguments before Theorem
\ref{mainthm1}.

 This completes the proof of Theorem \ref{mainthm2}.\end{proof}

\begin{proof}[Proof of Theorem \ref{mainthm5}]
Suppose that $\cal D$ is  a maximal $N-$set with the minimum size.
 By Theorem
\ref{mainthm2} we have $$d({\cal D})\subseteq \{ d : d\mid
p_{u+1}\cdots p_n\} .$$ Since no any two elements of ${\cal D}$ are
coprime, we know that no any two elements of $d({\cal D})$ are
coprime. That is (a). By the definition of $d({\cal D})$ we know
that no element of $d({\cal D})$ can be divided by another element
of $d({\cal D})$. That is (b). Let $l\mid p_{u+1}\cdots p_n$. If
$l\in {\cal D}$, then $l$ can be divided by an element of $d({\cal
D})$. If $l\notin {\cal D}$, then, by the maximality of $\cal D$,
there exists $d_1\in {\cal D}$ with $(d_1, l)=1$. Since $d_1\in
{\cal D}$, there exists $d\in d({\cal D})$ with $d\mid d_1$. Hence
$(d, l)=1$. That is (c). Hence $d({\cal D})$ is one of ${\cal T}_1,
\ldots , {\cal T}_k$.
 By Lemma \ref{lem5} we have ${\cal D}=R(d({\cal D}), N)$.  Hence ${\cal
D}$ is one of $R({\cal T}_1, N), \ldots , R({\cal T}_k, N)$.

Now we show that each $R({\cal T}_i,N)$ is a maximal $N-$set  with
the minimum size.

Since no any two elements of ${\cal T}_i$ are coprime, we know
that no any two elements of $R({\cal T}_i,N)$ are coprime. That
is, $R({\cal T}_i,N)$ is a $N-$set. In order to prove that
$R({\cal T}_i,N)$ is maximal, it is enough to prove that for any
$l>1$ with $l\mid N$ and $l\notin R({\cal T}_i,N)$  there exists
$d\in R({\cal T}_i,N)$ with $(d, l)=1$. It is enough to prove that
there exists $d\in {\cal T}_i$ with $(d, l)=1$. Let $l_1=(l,
p_{u+1}\cdots p_n)$. Noting that ${\cal T}_i$ is a set of positive
divisors of $p_{u+1}\cdots p_n$, it is enough to prove that there
exists $d\in {\cal T}_i$ with $(d, l_1)=1$. Since $l\notin R({\cal
T}_i,N)$, we know that $l$ cannot be divided by any element of
${\cal T}_i$. So $l_1$ cannot be divided by any element of ${\cal
T}_i$. By the definition of ${\cal T}_i$ (i. e. (c) of Theorem
\ref{mainthm5}), there exists $d\in {\cal T}_i$ with $(d, l_1)=1$.
Thus we have proved that $R({\cal T}_i,N)$ is a maximal $N-$set.
Noting that no element of ${\cal T}_i$ can be divided by another
element of ${\cal T}_i$, we have $d(R({\cal T}_i,N))={\cal T}_i$.
 Since ${\cal T}_i \subseteq \{ d : d \mid p_{u+1}\cdots p_n\}$, by  Theorem \ref{mainthm2} we have $R({\cal T}_i,N)$ has the
minimum size.
 This completes the proof of Theorem \ref{mainthm5}.\end{proof}

\section{Final Remarks}

Finally we pose the following problems for further research.

\begin{problem} Fix $t\ge 2$ and $N=p_1^{\alpha_1}\cdots p_n^{\alpha_n}$, $\alpha_1\ge
\alpha_2\ge\cdots \ge \alpha_n$. Let $\cal D$ be a set of positive
divisors $d$ of $N$ which have exactly $t$ distinct prime factors
(i.e. $\omega (d)=t$) such that no two members of the set being
coprime and no additional member may be included in $\cal D$ without
contradicting this requirement. Determine $m(N, t)=\min |{\cal
D}|$.\end{problem}

\begin{problem} Fix $t\ge 2$ and $N=p_1^{\alpha_1}\cdots p_n^{\alpha_n}$, $\alpha_1\ge
\alpha_2\ge\cdots \ge \alpha_n$. Let $\cal D$ be a set of positive
divisors $d$ of $N$ which have exactly $t$  prime factors (i.e.
$\Omega (d)=t$) such that no two members of the set being coprime
and no additional member may be included in $\cal D$ without
contradicting this requirement. Determine $M(N, t)=\min |{\cal D}|$.
\end{problem}

{\bf Acknowledgements. } We are grateful to the referees for their
valuable comments.

\end{document}